\documentclass[12pt]{article}
\usepackage{listings}
\usepackage{amsmath,amssymb}
\usepackage{subcaption}
\usepackage{graphicx}
\usepackage{tikz}
\usepackage{structuralanalysis}
\usepackage{siunitx}
\usepackage{enumerate}
\usepackage{mathtools}
\usepackage{epic}
\usepackage{float}
\usepackage{mathtools}
\usepackage{authblk}
\usepackage{blindtext}
\usepackage[numbers]{natbib}
\usepackage{enumitem}
 
\usepackage{geometry}
\usepackage[hang,flushmargin]{footmisc}

\newcommand{\qed}{\hfill \mbox{\raggedright \rule{.07in}{.1in}}}
\newenvironment{proof}{\vspace{1ex}\noindent{\bf Proof}\hspace{0.5em}}
	{\hfill\qed\vspace{1ex}}
\newtheorem{theorem}{Theorem}
\newtheorem{example}{Example}

\newtheorem{observation}{Observation}
\newtheorem{definition}{Definition}

\newtheorem{remark}{Remark}
\newtheorem{corollary}{Corollary}

\usepackage{graphicx}

\providecommand{\keywords}[1]{\textbf{\textit{keywords:}} #1}

\date{}
\begin{document}

	\title{Inverse of $\alpha$-Hermitian Adjacency Matrix of a Unicyclic Bipartite Graph}
	

	\author{Mohammad Abudayah \thanks{School of Basic Sciences and Humanities, German Jordanian University, mohammad.abudayah@gju.edu.jo }, Omar Alomari \thanks{School of Basic Sciences and Humanities, German Jordanian University, omar.alomari@gju.edu.jo}, Omar AbuGhneim \thanks{Department of Mathematics, Faculty of Science, The University of Jordan, o.abughneim@ju.edu.jo} 

		}

\maketitle

\begin{abstract}
Let $X$ be bipartite mixed graph and for a unit complex number $\alpha$, $H_\alpha$ be its $\alpha$-hermitian adjacency matrix. If $X$ has a unique perfect matching, then $H_\alpha$ has a hermitian inverse $H_\alpha^{-1}$. In this paper we give a full description of the entries of $H_\alpha^{-1}$ in terms of the paths between the vertices. Furthermore, for $\alpha$ equals the primitive third root of unity $\gamma$ and for a unicyclic bipartite graph $X$ with unique perfect matching, we characterize when $H_\gamma^{-1}$ is $\pm 1$ diagonally similar to $\gamma$-hermitian adjacency matrix of a mixed graph. Through our work, we have provided a new construction for the $\pm 1$ diagonal matrix. 
 
\end{abstract}


\section{\normalsize Introduction}
A partially directed graph $X$ is called a mixed graph, the undirected edges in $X$ are called digons and the directed edges are called arcs. Formally, a mixed graph $X$ is a set of vertices $V(X)$ together with a set of undirected edges $E_0(D)$ and a set of directed edges $E_1(X)$. For an arc $xy \in E_1(X)$, $x$(resp. $y$) is called initial (resp. terminal) vertex. The graph obtained from the mixed graph $X$ after stripping out the orientation of its arcs is called the underlying graph of $X$ and is denoted by $\Gamma(X)$.\\
A collection of digons and arcs of a mixed graph $X$ is called a perfect matching if they are vertex disjoint and cover $V(X)$. In other words, perfect matching of a mixed graph is just a perfect matching of its underlying graph. In general, a mixed graph may have more than one perfect matching. We denote the class of bipartite mixed graphs with a unique perfect matching by $\mathcal{H}$. In this class of mixed graphs the unique perfect matching will be denoted by $\mathcal{M}$. For a mixed graph $X\in \mathcal{H}$, an arc $e$ (resp. digon)  in $\mathcal{M}$ is called matching arc (resp. matching digon) in $X$.
 If $D$ is a mixed subgraph of $X$, then the mixed graph $X\backslash D$ is the induced mixed graph over $V(X)\backslash V(D)$.\\
Studying a graph or a digraph structure through properties of a matrix associated with it is an old and rich area of research.  
For undirected graphs, the most popular and widely investigated matrix in literature is the adjacency matrix. The adjacency matrix of a graph is symmetric, and thus diagonalizable and all of its eigenvalues are real.   On the other hand, the adjacency matrix of directed graphs and mixed graphs is not symmetric and its eigenvalues are not all real. Consequently, dealing with such matrix is very challenging. Many researchers have recently proposed other adjacency matrices for digraphs. For instance in \cite{Irena}, the author investigated the spectrum of $AA^T$, where $A$ is the traditional adjacency matrix of a digraph. The author called them non negative spectrum of digraphs. In \cite{OMT1}, authors proved that the non negative spectrum is totally controlled by a vertex partition called common out neighbor partition. Authors in \cite{BMI} and in \cite{LIU2015182} (independently) proposed a new adjacency matrix of mixed graphs as follows: 
For a mixed graph $X$, the hermitian adjacency matrix of $X$ is a $|V|\times |V|$ matrix $H(X)=[h_{uv}]$, where
 
\[h_{uv} = \left\{
\begin{array}{ll}
1 & \text{if }  uv \in E_0(X),\\
i &  \text{if }  uv \in E_1(X), \\
-i &  \text{if }  vu \in E_1(X),\\
0 & \text{otherwise}.
\end{array}
\right.
\]

This matrix has many nice properties. It has real spectrum and interlacing theorem holds. Beside investigating basic properties of this hermitian adjacency matrix, authors proved many interesting properties of the spectrum of $H$. This motivated Mohar in \cite{Mohar2019ANK} to extend the previously proposed adjacency matrix. The new kind of hermitian adjacency matrices, called $\alpha$-hermitian adjacency matrices of mixed graphs, are defined as follows: Let $X$ be a mixed graph and $\alpha$ be the primitive $n^{th}$ root of unity $e^{\frac{2\pi}{n}i}$. Then the $\alpha$ hermitian adjacency matrix of $X$ is a $|V|\times |V|$ matrix $H_{\alpha}(X)=[h_{uv}]$, where
 
\[h_{uv} = \left\{
\begin{array}{ll}
1 & \text{if }  uv \in E_0(D),\\
\alpha &  \text{if }  uv \in E_1(D), \\
\overline{\alpha} &  \text{if }  vu \in E_1(D),\\
0 & \text{otherwise}.
\end{array}
\right.
\]

Clearly the new kind of hermitian adjacency matrices of mixed graphs is a natural generalization of the old one for mixed graphs and even for the graphs. As we mentioned before these adjacency matrices ($H_i(X)$ and $H_\alpha(X)$) are hermitian and have interesting properties. This paved the way to more a facinating research topic much needed nowadays.\\
For simplicity when dealing with one mixed graph $X$, then we write $H_\alpha$ instead of $H_\alpha(X)$. \\\\

The smallest positive eigenvalue of a graph plays an important role in quantum chemistry. Motivated by this application,  Godsil in \cite{God} investigated the inverse of the adjacency matrix of a bipartite graph. He proved that if $T$ is a tree graph with perfect matching and $A(T)$ is its adjacency matrix then, $A(T)$ is invertabile and there is $\{1,-1\}$ diagonal matrix $D$ such that $DA^{-1}D$ is an adjacency matrix of another graph. Many of the problems mentioned in \cite{God} are still open.  Further research appeared after this paper that continued on Godsil's work see \cite{Pavlkov}, \cite{McLeman2014GraphI} and \cite{Akbari2007OnUG}.\\
 
In this paper we study the inverse of $\alpha$-hermitian adjacency matrix $H_\alpha$ of unicyclic bipartite mixed graphs with unique perfect matching $X$. Since undirected graphs can be considered as a special case of mixed graphs, the out comes in this paper are broader than the work done previously in this area.  We examine the inverse of $\alpha$-hermitian adjacency matricies of bipartite mixed graphs and unicyclic bipartite mixed graphs. Also, for $\alpha=\gamma$, the primative third root of unity, we answer the traditional question, when $H_\alpha^{-1}$ is $\{\pm 1\}$ diagonally similar to an $\alpha$-hermitian adjacency matrix of mixed graph. To be more precise, for a unicyclic bipartite mixed graph $X$ with unique perfect matching we give full characterization when there is a $\{\pm 1\}$ diagonal matrix $D$ such that $DH_\gamma^{-1}D$ is an $\gamma$-hermitian adjacency matrix of a mixed graph. Furthermore,  through our work we introduce a construction of such diagonal matrix $D$. In order to do this, we need the following definitions and theorems:
\begin{definition}\citep{Abudayah2}
Let $X$ be a mixed graph and $H_\alpha=[h_{uv}]$ be its $\alpha$-hermitian adjacency matrix.
\begin{itemize}
\item $X$ is called elementary mixed graph if for every component $X'$ of $X$, $\Gamma(X')$ is either an edge or a cycle $C_k$ (for some $k\ge 3$).
\item For an elementary mixed graph $X$, the rank of $X$ is defined as $r(X)=n-c,$ where $n=|V(X)|$ and $c$ is the number of its components. The co-rank of $X$ is defined as $s(X)=m-r(X)$, where $m=|E_0(X)\cup E_1(X)|$.
\item For a mixed walk $W$  in $X$, where $\Gamma(W)=r_1,r_2,\dots r_k$, the value $h_\alpha(W)$ is defined as  $$h_\alpha(W)=h_{r_1r_2}h_{r_2r_3}h_{r_3r_4}\dots h_{r_{k-1}r_k}\in \{\alpha^n\}_{n\in \mathbb{Z}}$$
\end{itemize}
\end{definition}

Recall that a bijective function $\eta$ from a set $V$ to itself is called  permutation. The set of all permutations of a set $V$, denoted by $S_V$, together with functions composition form a group. Finally recall that for $\eta \in S_V$, $\eta$ can be written as composition of transpositions. In fact the number of transpositions is not unique. But this number is either odd or even and cannot be both. Now, we define $sgn(\eta)$ as $(-1)^k$, where $k$ is the number of transposition when $\eta$ is decomposed as a product of transpositions. The following theorem is  well known as a classical result in linear algebra 
\begin{theorem} \label{exp}
If $A=[a_{ij}]$ is an $n\times n$ matrix then $$det(A)=\displaystyle \sum_{\eta \in S_n } sgn(\eta) a_{1,\eta(1)}a_{2,\eta(2)}a_{3,\eta(3)}\dots a_{n,\eta(n)}  $$
\end{theorem}

\section{Inverse of $\alpha$-hermitian adjacency matrix of a bipartite mixed graph}
In this section, we investigate the invertibility of the $\alpha$-hermitian adjacency matrix of a bipartite mixed graph $X$. Then we find a formula for the entries of its inverse based on elementary mixed subgraphs. This will lead to a formula for the entries based on the type of the paths between vertices.
Using Theorem \ref{exp}, authors in \cite{Abudayah2} proved the following theorem.

\begin{theorem}(Determinant expansion for $H_{\alpha}$) \cite{Abudayah2} \label{Determinant}
Let $X$ be a mixed graph and $H_\alpha$ its $\alpha$-hermitian adjacency matrix, then
$$ det( H_{\alpha}) = \sum_{X'} (-1)^{r(X')}2^{s(X')}Re  \left(\prod_C h_{\alpha} ( \vec{C} )\right) $$ 
where the sum ranges over all spanning elementary mixed subgraphs $X'$ of $X$, the product ranges over all mixed cycles $C$ in $X'$, and $\vec{C}$ is any mixed closed walk traversing $C$.
\end{theorem}

Now, let $X\in \mathcal{H}$ and $\mathcal{M}$ is the unique perfect matching in $X$. Then since $X$ is bipartite graph, $X$ contains no odd cycles. Now, let $C_k$ be a cycle in $X$, then if $C_k \cap \mathcal{M}$ is a perfect matching of $C_k$ then, $\mathcal{M} \Delta C_k= \mathcal{M}\backslash C_k \cup  C_k \backslash \mathcal{M}$ is another perfect matching in $X$ which is a contradiction. Therefore there is at least one vertex of $C_k$ that is matched by a matching edge not in $C_k$. This means if $X\in \mathcal{H}$, then $X$ has exactly one spanning elementary mixed subgraph that consist of only $K_2$ components. Therefore, Using the above discussion together with Theorem \ref{Determinant} we get the following theorem.

\begin{theorem}\label{Inv}
If $X\in \mathcal{H}$ and $H_\alpha$ is its $\alpha$-hermitian adjacency matrix then $H_\alpha$ is non singular.
\end{theorem}

Now, Let $X$ be a mixed graph and $H_\alpha$ be its $\alpha$-hermitian adjacency matrix. Then, for invertible $H_\alpha$, the following theorem finds a formula for the entries of $H_\alpha^{-1}$ based on elementary mixed subgraphs and paths between vertices. The proof can be found in \cite{invtree}.

\begin{theorem}\label{Thm1}
Let $X$ be a mixed graph, $H_\alpha$ be its $\alpha$-hermitian adjacency matrix and for $i \neq j$, $\rho_{i \to j}=\{ P_{i \to j}: P_{i \to j} \text{ is a mixed path from the vertex } i \text{ to the vertex } j \}$. If $\det(H_\alpha) \ne 0$, then
\begin{align*}
	[H_\alpha^{-1}]_{ij} =&\\
	& \frac{1}{\det(H_\alpha)}\displaystyle \sum_{P_{i \to j}\in \rho_{i \to j}} (-1)^{|E(P_{i \to j})|} \text{ } h_\alpha (P_{i \to j})  \sum_{X'} (-1)^{r(X')} 2^{s(X')} Re \left(  \prod_C h_\alpha (\vec{C})\right)    
\end{align*}
where the second sum ranges over all spanning elementary mixed subgraphs $X'$ of $X\backslash P_{i \to j}$, the product is being taken over all mixed cycles $C$ in $X'$ and $\vec{C}$ is any mixed closed walk traversing $C$.
\end{theorem}
This theorem describes how to find the non diagonal entries of $H_\alpha^{-1}$. In fact, the diagonal entries may or may not equal to zero. To observe this, lets consider the following example:

\begin{example}
Consider the mixed graph $X$ shown in Figure \ref{fig:A} and let $\alpha=e^{\frac{\pi}{5}i}$. The mixed graph $X$ has a unique perfect matching, say $M$, and this matching consists of the set of unbroken arcs and digons. Further $M$ is the unique spanning elementary mixed subgraph of $X$. Therefore, using Theorem \ref{Determinant}
\[ det[H_\alpha]= (-1)^{8-4}2^{4-4}=1
\] 
So, $H_\alpha$ is invertible. To calculate $[H_\alpha^{-1}]_{ii}$, we observe that 
\[ [H_\alpha^{-1}]_{ii}= \frac{det((H_\alpha)_{(i,i)})}{det(H_\alpha)}=det((H_\alpha)_{(i,i)}).
\]
Where $(H_\alpha)_{(i,i)}$ is the matrix obtained from $H_\alpha$ by deleting the $i^{th}$ row and $i^{th}$ column, which is exactly the $\alpha$-hermitian adjacency matrix of $X\backslash \{i\}$. Applying this on the mixed graph, one can deduce that the diagonal entries of $H_\alpha^{-1}$ are all zeros except the entry $(H_\alpha^{-1})_{11}$. In fact it can be easily seen that the mixed graph $X \backslash \{1\}$ has only one spanning elementary mixed subgraph. Therefore,
\[ [H_\alpha^{-1}]_{11}=det((H_\alpha)_{(1,1)})=(-1)^{7-2}2^{6-5}Re(\alpha)=-2Re(\alpha).
\]
\begin{figure}[ht]
\centering
\includegraphics[width=0.8\linewidth]{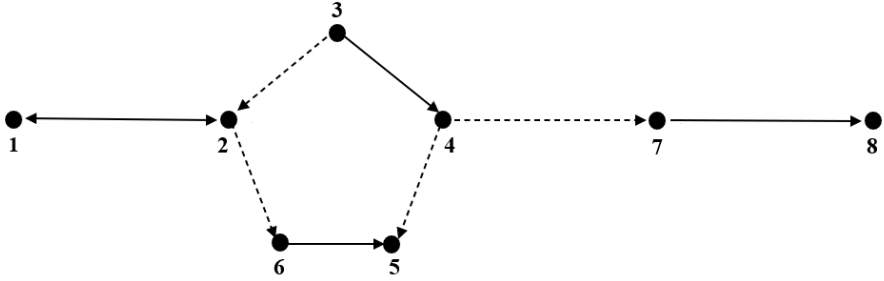}
\caption{Mixed Graph $X$ where $H_\alpha^{-1}$ has nonzero diagonal entry}
\label{fig:A}
\end{figure}

\end{example}
The following theorem shows that if $X$ is a bipartite mixed graph with unique perfect matching, then the diagonal entries of $H_\alpha^{-1}$ should be all zeros. 
\begin{theorem}

Let $X \in \mathcal{H}$ and $H_\alpha$ be its $\alpha$-hermitian adjacency matrix. Then, for every vertex $i \in V(X)$, $(H_\alpha^{-1})_{ii} =0$.
\end{theorem}

\begin{proof}
Observing that $X$ is a bipartite mixed graph with a unique perfect matching, and using Theorem \ref{Inv}, we have $H_\alpha$ is invertable. Furthermore,
\[
	(H_\alpha^{-1})_{ii} = \frac{\det((H_\alpha)_{(i,i)})}{\det(H_\alpha)}
\]

Note that $(H_\alpha)_{(i,i)}$ is the $\alpha$-hermitian adjacency matrix of the mixed graph $X\backslash \{i\}$. However $X$ has a unique perfect matching, therefore $X\backslash \{i\}$ has an odd number of vertices. Hence $X\backslash \{i\}$ has neither a perfect matching nor an elementary mixed subgraph and thus $\det((H_\alpha)_{(i,i)})=0$.
\end{proof}\\
Now, we investigate the non diagonal entries of the inverse of the $\alpha$-hermitian adjacency matrix of a bipartite mixed graph, $X \in \mathcal{H}$. In order to do that we need to characterize the structure of the mixed graph $X \backslash P$ for every mixed path $P$ in $X$. To this end, consider the following theorems:

\begin{theorem}\cite{clark1991first}\label{clark}
Let $M$ and $M'$ be two matchings in a graph $G$. Let $H$ be the subgraph of $G$ induced by the set of edges $$M \Delta M'=(M\backslash M') \cup (M' \backslash M).$$
Then, the components of $H$ are either cycles of even number of vertices whose edges alternate in $M$ and $M'$ or a path whose edges alternate in $M$ and $M'$ and end vertices unsaturated in one of the two matchings.
\end{theorem}

\begin{corollary} \label{c1}
For any graph $G$, if $G$ has a unique perfect matching then $G$ does not contain alternating cycle.
 \end{corollary}

\begin{definition}
Let $X$ be a mixed graph with unique perfect matching. A path $P$ between two vertices $u$ and $v$ in $X$ is called co-augmenting path if the edges of the underlying path of $P$ alternates between matching edges and non-matching edges where both first and last edges of $P$ are matching edges.
\end{definition}
\begin{corollary} \label{c2}
Let $G$ be a bipartite graph with unique perfect matching $\mathcal{M}$, $u$ and $v$ are two vertices of $G$. If $P_{uv}$ is a co-augmenting path between $u$ and $v$, then $G \backslash P_{uv}$ is a bipartite graph with unique perfect matching $\mathcal{M}\backslash P_{uv}$.
 \end{corollary}
\begin{proof}
The part that $\mathcal{M}\backslash P_{uv}$ is being a perfect matching of $G \backslash P_{uv}$ is obvious.
Suppose that $M' \ne \mathcal{M}\backslash P_{uv}$ is another perfect matching of  $G \backslash P_{uv}$. Using Theorem \ref{clark},  $G \backslash P_{uv}$ consists of an alternating cycles or an alternating paths, where its edges alternate between $\mathcal{M}\backslash P_{uv}$ and $M'$. If all $G \backslash P_{uv}$ components are paths, then $G \backslash P_{uv}$ has exactly one perfect matching, which is a contradiction. Therefore, $G \backslash P_{uv}$ contains an alternating cycle say $C$. Since $P_{uv}$ is a co-augmenting path, we have $M' \cup (P_{uv} \cap \mathcal{M})$ is a perfect matching of $G$. Therefore $G$ has more than one perfect matching, which is a contradiction.
\end{proof}\\
\begin{theorem}\label{nco}
Let $G$ be a bipartite graph with unique perfect matching $\mathcal{M}$, $u$ and $v$ are two vertices of $G$. If $P_{uv}$ is not a co-augmenting path between $u$ and $v$, then $G \backslash P_{uv}$ does not have a perfect matching.
\end{theorem}
\begin{proof}
Since $G$ has a perfect matching, then $G$ has even number of vertices. Therefore, when $P_{uv}$ has an odd number of vertices, $G \backslash P_{uv}$ does not have a perfect matching.\\
Suppose that $P_{uv}$ has an even number of vertices. Then, $P_{uv}$ has a perfect matching $M$. Therefore if  $G \backslash P_{uv}$ has a perfect matching $M'$, then $M \cup M'$ will form a new perfect matching of $G$. This contradicts the fact that $G$ has a unique perfect matching. 
\end{proof}\\

Now, we are ready to give a formula for the entries of the inverse of $\alpha$-hermitian adjacency matrix of bipartite mixed graph $X$ that has a unique perfect matching. This characterizing is based on the co-augmenting paths between vertices of $X$.

\begin{theorem}\label{Thm2}
Let $X$ be a bipartite mixed graph with unique perfect matching $\mathcal{M}$, $H_\alpha$ be its $\alpha$-hermitian adjacency matrix and
$$\Im_{i \to j}=\{ P_{i \to j}: P_{i \to j} \text{\small{ is a co-augmenting mixed path from the vertex }} i \text{ to the vertex } j \}$$ Then

\[ 
(H_\alpha^{-1})_{ij}= \left\{
\begin{array}{ll}
\displaystyle \sum_{P_{i\to j} \in \Im_{i\to j}} (-1)^{\frac{|E(P_{i \to j})|-1}{2}} h_\alpha(P_{i \to j})  & \text{if }  i\ne j    \\
0  & \text{ if } i =j
\end{array}
\right.
\]

\end{theorem}

\begin{proof}

Using Theorem \ref{Thm1},
$${ [H_{\alpha}^{-1}]_{ij} = \frac{1}{\det(H_\alpha)} \sum_{P_{i \rightarrow j} \in \rho_{i \rightarrow j}} \left[ (-1)^{|E(P_{i \rightarrow j})|} h_\alpha(P_{i \rightarrow j})  \sum_{X'} (-1)^{r(X')} 2^{s(X')} Re  (\prod_C h_{\alpha} ( \vec{C} ))  \right ]} $$

where the second sum ranges over all spanning elementary mixed subgraphs of $X \backslash P_{i \rightarrow j}$. The product is being taken over all mixed cycles $C$ of $X'$ and $\vec{C}$ is any mixed closed walk traversing $C$. \\

First, using Theorem \ref{nco} we observe that if $P_{i \rightarrow j}$ is not a co-augmenting path then $X \backslash P_{i\to j}$ does not have a perfect matching. Therefore, the term corresponds to $P_{i\to j}$ contributes zero. Thus we only care about the co-augmenting paths.
According to Corollary \ref{c2}, for any co-augmenting path  $P_{i\to j}$ from the vertex $i$ to the vertex $j$ we get  $X \backslash P_{i\to j}$ has a unique perfect matching, namely $\mathcal{M}\cap E( X \backslash P_{i\to j})$. Using Corollary \ref{c1},  $X \backslash P_{i\to j}$ does not contain an alternating cycle. Thus  $X \backslash P_{i\to j}$ contains only one spanning elementary mixed subgraph which is  $\mathcal{M} \backslash P_{i\to j}$. So,

$$ [H_{\alpha}^{-1}]_{ij} = \frac{1}{\det(H_\alpha)} \sum_{P_{i \to j} \in  \Im_{i\to j}} (-1)^{|E(P_{i \to j})|} h_\alpha(P_{i \to j})   (-1)^{V(X\backslash P_{i \to j})-k}  $$ 

where $k$ is the number of components of the spanning elementary mixed subgraph of $X \backslash P_{i\rightarrow j}$.
Observe that $| V(X \backslash P_{i\rightarrow j})|=n-(|E(P_{i \rightarrow j})|+1)$, $k=\frac{n-(|E(P_{i\rightarrow j})|+1)}{2}$ and $\det(H_\alpha) = (-1)^\frac{n}{2}$, we get the result.
\end{proof}\\

\section{Inverse of $\gamma$-hermitian adjacency matrix of a unicyclic bipartite mixed graph}
\begin{figure}[ht]
\centering
\includegraphics[width=0.6\linewidth]{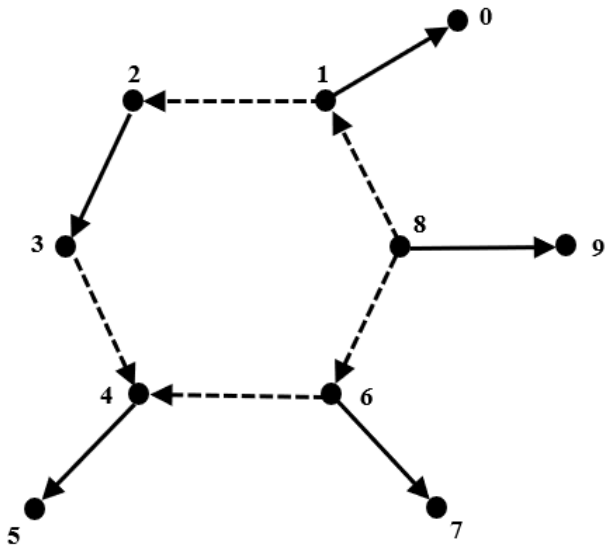}
\caption{Unicycle bipartite mixed graph with unique perfect matching and $4$ pegs  }
\label{fig:D}
\end{figure}
Let $\gamma$ be the third root of unity $e^{\frac{2\pi}{3}i}$. Using Theorem \ref{Thm2}, $h_\alpha(P_{i\to j})\in \{\alpha^i\}_{i=1}^n$ plays a central rule in finding the entries of $H_\alpha^{-1}$ and since the third root of unity has the property $\gamma^i \in \{1,\gamma, \overline{\gamma}\}$ we focus our study in this section on $\alpha=\gamma$. The property that $\alpha^i \in \{\pm1, \pm \alpha, \pm \overline{\alpha}\}$ is not true in general. To illustrate, consider the mixed graph shown in Figure \ref{fig:D} and let $\alpha=e^{\frac{\pi}{5}i}$. Using Theorem \ref{Thm2} we get $H_{05}^{-1}=e^{\frac{3\pi}{5}i}$ which is not from the set $\{\pm 1, \pm \alpha, \pm \overline{\alpha}\}$.\\
In this section, we are going to answer the classical question whether the inverse of $\gamma$-hermitian adjacency matrix of a unicyclic bipartite graph is $\{1,-1\}$ diagonally similar to a hermitian adjacency matrix of another mixed graph or not. Consider the mixed graph shown in Figure \ref{fig:D}. Then, obviously entries of $H_\gamma^{-1}$ are from the set $\{0,\pm 1, \pm \gamma, \pm \overline{\gamma}\}$ \\
Another thing we should bear in mind is the existence of $\{1,-1\}$ diagonal matrix $D$ such that $DH_\gamma D$ is $\gamma$-adjacency matrix of another mixed graph. In the mixed graph $X$ in Figure \ref{fig:D}, suppose that $D=diag\{d_{0},d_{1},\dots,d_{9}\}$ is $\{1,-1\}$ diagonal matrix with the property $DH_\gamma D$ has all entries from the set $\{0, \gamma, \overline{\gamma}\}$. Then, \\
\[
\begin{array}{l}
 d_0d_5=1   \\
 d_0d_9=-1 \\
 d_9d_7=-1  \\
 d_5d_7=-1
\end{array}
\]
which is impossible. Therefore, such diagonal matrix $D$ does not exist. To discuss the existence of the diagonal matrix $D$ further, let $G$ be a bipartite graph with unique perfect matching. Define $X_G$ to be the mixed graph obtained from $G$ by orienting all non matching edges. Clearly for $\alpha \ne 1$ and $\alpha \ne -1$ changing the orientation of the non matching edges will change the $\alpha$-hermitian adjacency matrix. For now lets restrict our study on $\alpha=-1$. Using Theorem \ref{Thm2} one can easily get the following observation.

\begin{observation}\label{corr1}
Let $G$ be a bipartite mixed graph with unique perfect matching $\mathcal{M}$, $H_{-1}$ be the $-1$-hermitian adjacency matrix of $X_G$ and
$$\Im_{i \to j}=\{ P_{i \to j}: P_{i \to j} \text{ is a co-augmenting mixed path from the vertex } i \text{ to the vertex } j \text{ in } X_G \}.$$ 
One can use Theorem \ref{Thm2} to get

\[ 
(H_{-1}^{-1})_{ij}= \left\{
\begin{array}{ll}
\displaystyle \vert \Im_{i\to j} \vert   & \text{if }  i\ne j    \\
0  & \text{ if } i =j
\end{array}
\right.
\]
\end{observation}

So, the question we need to answer now is when $A(G)$ and $H_{-1}(X_G)$ are diagonally similar. To this end, let $G$ be a bipartite graph with a unique perfect matching and $u\in V(G)$. Then for a walk $W=u=r_1,r_2,r_3,\dots,r_k$ in $G$, define a function that assign the value $f_W(j)$ for the $j^{th}$ vertex of $W$ as follows:
\[f_W(1)=1\]
and
\[
f_W(j+1)= \left\{
\begin{array}{ll}
-f_W(j)   & \text{if }  r_jr_{j+1} \text{is unmatching edge in } G     \\
f_W(j)  & \text{if }  r_jr_{j+1} \text{is matching edge in } G
\end{array}
\right.
\]

See Figure \ref{fig:E}.
\begin{figure}[ht]
\centering
\includegraphics[width=0.6\linewidth]{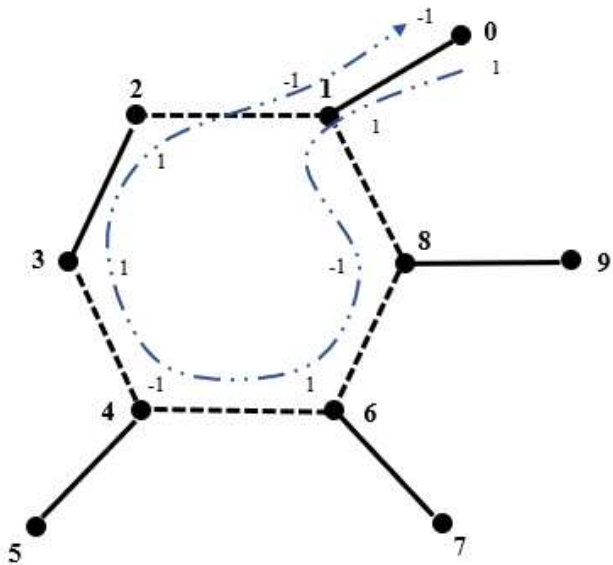}
\caption{The values of $f_W$ where $W$ is the closed walk starting from $0$  }
\label{fig:E}
\end{figure}
Since any path from a vertex $u$ to itself consist of pairs of identical paths and cycle, we get the following remark.
\begin{remark}
Let $G$ be bipartite graph with unique perfect matching and $F(u)=\{f_W(u): W \text{ is a closed walk in } G \text{ starting at u}\}.$ then, $\vert F(u) \vert=1$ if and only if the number of unmatching edges in each cycle of $G$ is even.
\end{remark}

Finally, let $G$ be a bipartite graph with unique perfect matching and suppose that each cycle of $G$ has an even number of unmatched edges. For a vertex $u\in V(G)$ define the function $w:V(G) \to \{1,-1\}$ by 
\[ w(v)=f_W(v), \text{  where } W \text{ is a path from } u \text{ to } v
\]
\begin{definition}
Suppose that $G$ is bipartite graph with unique perfect matching and every cycle of $G$ has even number of unmatched edges. Suppose further $V(G)=\{v_1,v_2,\dots,v_n\}$ and $u\in V(G)$. Define the matrix $D_u$ by $D_u=diag\{w(v_1),w(v_2),\dots,w(v_n)\}$.
\end{definition}
\begin{theorem}\label{her}
Suppose $G$ is a bipartite graph with unique perfect matching and every cycle of $G$ has an even number of unmatched edges. Then for every $u \in V(G)$, we get $D_uA(G)D_u=H_{-1}(X_G)$. 
\end{theorem}
\begin{proof}
Note that, for $x,y \in V(G)$, we have $(D_uA(G)D_u)_{xy}=d_xa_{xy}d_y$. Using the definition of $D_u$ we get,

\[
d_xd_y= \left\{
\begin{array}{ll}
-1   & \text{if }  xy \text{ is an unmatching edge in } G     \\
1  & \text{if }  xy \text{ is a matching edge in } G\\
0 & \text{ otherwise }
\end{array}
\right.
\]
Therefore, $(D_uA(G)D_u)_{xy}=(H_{-1})_{xy}$.

\end{proof}\\

Now we are ready to discuss the inverse of $\gamma$-hermitian adjacency matrix of unicyclic mixed graph. Let $X$ be a unicyclic bipartite graph with unique perfect matching. An arc or digon of $X$ is called a peg if it is a matching arc or digon and incident to a vertex of the cycle in $X$. Since $X$ is unicyclic bipartite graph with unique perfect matching, $X$ has at least one peg. Otherwise the cycle in $X$ will be alternate cycle, and thus $X$ has more than one perfect matching which contradicts the assumption. Since each vertex of the cycle incident to a matching edge and $|V(X)|$ is even, $X$ should contain at least two pegs. The following theorem will deal with unicyclic bipartite mixed graphs $X\in \mathcal{H}$ with more than two pegs.

\begin{theorem}\label{peg}
Let $X$ be a unicyclic bipartite graph with unique perfect matching. If $X$ has more than two pegs, then between any two vertices of $X$ there is at most one co-augmenting path. 
\end{theorem}

\begin{proof}
Let $\rho_1, \rho_2$ and $\rho_3$ be three pegs in $X$,  $u,v \in V(D)$, $C$ is the unique cycle in $X$ and suppose there are two co-augmenting paths between $u$ and $v$, say $P$ and $P'$. Since $X$ is unicyclic, we have $V(C) \subset P \cup P'$,

Case1: $E(P) \cup E(P')$ does not contain any of the pegs. Then, if $v$ is the $X$ cycle vertex incident to $\rho_1$ then, $v$ is not matched by an edge in the cycle, which means one of $P$ or $P'$ is not co-augmenting path, which contradicts the assumption.

Case2: $(E(P) \cup E(P')$ contain pegs. Then, $(E(P) \cup E(P')$ should contain at most two pegs, suppose that $\rho_1$ and $v$ is the vertex of $X$ cycle that incident to $\rho_1$. Then, $v$ belongs to either $P$ or $P'$, again since $\rho_1$ is a matched edge, $v$ is not matched by the cycle edges which means one of $P$ or $P'$ is not co-augmenting path. which contradicts the assumption.

\end{proof}

\begin{corollary}\label{p1}
Let $X$ be a unicycle bipartite mixed graph with unique perfect matching. If $X$ has more than two pegs, then
\begin{enumerate}
\item $
(H_\alpha^{-1})_{ij}= \left\{
\begin{array}{ll}
(-1)^{\frac{|E(P_{i \to j})|-1}{2}} h_\alpha(P_{i \to j})  & \text{if }  P_{i\rightarrow j} \text{ is a co-augmenting path from } i \text{ to } j    \\
0  & \text{ Otherwise }
\end{array}
\right.
$
\item If the cycle of $X$ contains even number of unmatching edges, then for any vertex $u\in V(X)$, $D_uH^{-1}_\gamma(X)D_u$ is $\gamma$-hermitian adjacency matrix of a mixed graph.
\end{enumerate}
\end{corollary}
\begin{proof}
Part one is obvious using Theorem \ref{Thm2} together with Theorem \ref{peg}.\\
For part two, we observe that $\gamma^i\in \{1,\gamma,\overline{\gamma}\}$. Therefore all $H^{-1}_\gamma(X)$ entries are from the set $\{\pm 1,\pm \gamma,\pm \overline{\gamma}\}$. Also the negative signs in $A(\Gamma(X))^{-1}$ and in $H_\gamma^{-1}$ appear at the same position. Which means $D_uH_\gamma^{-1}D_u$ is $\gamma$-hermitian adjacency matrix of a mixed graph if and only if $D_uA(\Gamma(X))D_u$ is adjacency matrix of a graph. Finally, Theorem \ref{her} ends the proof.
\end{proof}

Now we will take care of unicycle graph with exactly two pegs.
Using the same technique of the proof of Theorem \ref{peg}, one can show the following:

\begin{theorem}\label{peg2}
	Let $D$ be a unicyclic bipartite graph with unique perfect matching and exactly two pegs $\rho_1$ and $\rho_2$. Then for any two vertices of $D$, $u$ and $v$, if there are two co-augmenting paths from the vertex $u$ to the vertex $v$, then $\rho_1$ and $\rho_2$ are edges of the two paths.
\end{theorem}
	
Let $X$ be a unicyclic bipartite mixed graph with unique perfect matching and exactly two pegs, and let $uv$ and $u'v'$ be the two pegs of $X$ where $v$ and $v'$ are vertices of the cycle of $X$. We, denote the two paths between $v$ and $v'$ by $\mathcal{F}_{v\rightarrow v'}$ and $\mathcal{F}_{v\rightarrow v'}^c$.

\begin{theorem}\label{two pegs}

Let $X$ be a unicyclic bipartite mixed graph with unique perfect matching and exactly two pegs and let $C$ be the cycle of $X$. If there are two coaugmenting paths between the vertex $x$ and the vertex $y$, then 

\[
(H_\alpha^{-1})_{xy}= (-1)^{\frac{|E(P_{x \to y})|-1}{2}} \frac{h_\alpha(P_{x \to v}) h_\alpha(P_{y \to v'}) }{h_\alpha(\mathcal{F}_{v \to v'})} \left[ (-1)^{m+1} h_\alpha(C)+1 \right] 
\]

where $\mathcal{F}_{v \to v'}$ is chosen to be the part of the path $P_{x \to y}$ in the cycle $C$ and $C$ is of size $2m$. 
\end{theorem}

\begin{proof}

Suppose that $P_{x \to y}$ and $Q_{x \to y}$ are the paths between the vertices $x$ and $y$, using theorem \ref{Thm2} we have 

\[
(H_\alpha^{-1})_{xy}= (-1)^{\frac{|E(P_{x \to y})|-1}{2}} h_\alpha(P_{x \to y}) + (-1)^{\frac{|E(Q_{x \to y})|-1}{2}} h_\alpha(Q_{x \to y}) 
\]

Now, using Theorem \ref{peg2}, $P_{x \to y}$ $(Q_{x \to y})$ can be divided into three parts $P_{x \to v}$, $\mathcal{F}_{v \to v'}$ and $P_{v' \to y}$ (resp.  $Q_{x \to v}=P_{x \to v},\text{ } \mathcal{F}_{v \to v'}^c$ and $Q_{v' \to y}=P_{v' \to y}$).\\ Observe that the number of unmatched edges in $\mathcal{F}_{v \to v'}$ is $k_1=\frac{|E(\mathcal{F}_{v \to v'})|+1}{2}$ and the number of unmatched edges in $\mathcal{F}_{v \to v'}^c$ is $k_2=m-\frac{|E(\mathcal{F}_{v \to v'})|+1}{2}+1$ we get 

\[
(H_\alpha^{-1})_{xy}=(-1)^k h_\alpha(P_{x \to v}) h_\alpha(P_{v \to y}) \left(  (-1)^{k_1}  h_\alpha(\mathcal{F}_{v \to v'}) +  (-1)^{k_2}  h_\alpha(\mathcal{F}_{v \to v'}^c)  \right)
\]
where $k=\frac{|E(P_{x \to v})|+|E(P_{v' \to y})|}{2}-1$

Note here $\overline{ h_\alpha(\mathcal{F}_{v \to v'})} h_\alpha(\mathcal{F}_{v \to v'}^c) = h_\alpha(C)$, therefore, 
\[
(H_\alpha^{-1})_{xy}= (-1)^{\frac{|E(P_{x \to y})|-1}{2}} \frac{h_\alpha(P_{x \to v}) h_\alpha(P_{y \to v'}) }{h_\alpha(\mathcal{F}_{v \to v'})} \left[ (-1)^{m+1} h_\alpha(C)+1 \right] 
\]

\end{proof}

\begin{theorem}\label{NH}
Let $X$ be a unicyclic bipartite mixed graph with unique perfect matching and $H_\gamma$ be its $\gamma$-hermitian adjacency matrix. If $X$ has exactly two pegs, then $H_\gamma^{-1}$ is not $\pm 1$ diagonally similar to $\gamma$-hermitian adjacency matrix of a mixed graph.
\end{theorem}
\begin{proof}
Let $xx'$ and $yy'$ be the two pegs of $X$, where $x'$ and $y'$ are vertices of the cycle $C$ of $X$ . Then, using Theorem \ref{two pegs} we have
\[(H_\gamma^{-1})_{xy}= (-1)^{\frac{|E(P_{x \to y})|-1}{2}} \frac{h_\gamma(P_{x \to x'}) h_\gamma(P_{y \to y'}) }{h_\gamma(\mathcal{F}_{x' \to y'})} \left[ (-1)^{m+1} h_\gamma(C)+1 \right] 
\]

where $\mathcal{F}_{x' \to y'}$ is chosen to be the part of the path $P_{x \to y}$ in the cycle $C$ and $C$ is of size $2m$. Suppose that $D=diag\{d_v:v\in V(X)\}$ is a $\{\pm 1\}$ diagonal matrix with the property that $DH_\gamma^{-1}D$ is $\gamma$-hermitian adjacency matrix of a mixed graph.

\begin{itemize}
\item Case1: Suppose $m$ is even say $m=2r$.\\
Observe that $(-1)^{m+1}h_\gamma(C)+1=1-h_\gamma(C)$. If $h_\gamma(C) \in \{1, \gamma, \gamma^2\}$, then $1-h_\gamma(C) \notin \{\pm 1, \pm \gamma, \pm \gamma^2\}$ and so $H_\gamma^{-1}$ is not $\pm 1$ diagonally similar to $\gamma$-hermitian adjacency matrix of a mixed graph. Thus we only need to discuss the case when $h_\gamma(C)=1$. To this end, suppose that $h_\gamma(C)=1$. Then $(H_\gamma^{-1})_{xy}=0$. Since the length of $C$ is $4r$, we have the number of unmatching edges (number of matching edges) in $C$ is $\frac{4r+2}{2}$ (resp. $\frac{4r-2}{2}$). Since the number of unmatching edges in $C$ is odd, there is a coaugmenting  path $\mathcal{F}_{x \to y}$ from $x$ to $y$ that contains odd number of unmatching edges and another coaugmenting path $\mathcal{F}^c_{x \to y}$ with even number of unmatching edges. Now, let $o'o$($e'e$ )  be any matching edges in the path $\mathcal{F}_{x \to y}$ (resp. $\mathcal{F}^c_{x \to y}$). Then, without loss of generality we may assume that there is a coaugmenting path between $x$ and $e$, $x$ and $o$ (and hence there is a co-augmenting path between $y$ and $o'$, $y$ and $e'$ ). Now, if $d_xd_y=1$ then 
\begin{itemize}
\item $(DH_\gamma^{-1}D)_{xo}=(-1)^kd_xh_\gamma(P_{x \to o})d_o$
\item $(DH_\gamma^{-1}D)_{yo'}=(-1)^{k'}d_yh_\gamma(P_{y \to o'})d_{o'}$
\end{itemize}
Observe that $k+k'$ is odd number, we have $d_od_{o'}=-1$. This contradict the fact that for every matching edge $gg'$, $d_gd_{g'}=1$.\\
The case when $d_xd_y=-1$ is similar to the above case but with considering the path $\mathcal{F}^c_{x \to y}$ instead of $\mathcal{F}_{x \to y}$ and the vertex $e$ instead of $o$.

\item Case2: Suppose $m$ is odd say $2r+1$. Then
\[
(H_\gamma^{-1})_{xy}= (-1)^{\frac{|E(P_{x \to y})|-1}{2}} \frac{h_\alpha(P_{x \to v}) h_\alpha(P_{y \to v'}) }{h_\alpha(\mathcal{F}_{v \to v'})} \left[ h_\alpha(C)+1 \right] .
\]
Therefore,
\[(H_\gamma^{-1})_{xy}= (-1)^{\frac{|E(P_{x \to y})|-1}{2}} \frac{h_\alpha(P_{x \to v}) h_\alpha(P_{y \to v'}) }{h_\alpha(\mathcal{F}_{v \to v'})}\left\{
\begin{array}{lll}
-\gamma   & \text{if }  h_\alpha(C)=\gamma^2  \\
-\gamma^2  & \text{if }  h_\alpha(C)=\gamma\\
2  & \text{if }  h_\alpha(C)=1

\end{array}
\right.
\]
Obviously, when $h_\alpha(C)=1$, $H_\gamma^{-1}$ is not $\pm 1$ diagonally similar to $\gamma$-hermitian adjacency matrix of a mixed graph. Thus, the cases we need to discuss here are when $h_\alpha(C)=\gamma$ and $h_\alpha(C)=\gamma^2$.\\
Since $m$ is odd, then $C$ contains an even number of unmatched edges. Therefore, either both paths between $x$ and $y$, $\mathcal{F}_{x\to y}$ and $\mathcal{F}_{x\to y}^c$, contain odd number of unmatching edges or both of them contains even number of unmatching edges. \\
To this end, suppose that both of the paths $\mathcal{F}_{x\to y}$ and $\mathcal{F}_{x\to y}^c$ contain odd number of unmatched edges. Then, $(H_\gamma^{-1})_{xy}\in \{\gamma^i\}_{i=0}^2$, which means $d_xd_y=1$.
Finally, let $v'v$ be any matching edge in $\mathcal{F}_{x\to y}$ where $P_{x \to v}$ and $P_{v' \to y}$ are coaugmenting paths, then obviously $d_vd_{v'}=1$. But one of the coaugmenting paths $P_{x \to v}$ and $P_{v' \to y}$ should contain odd number of unmatching edges and the other one should contain even number of unmatched edges. Which means $d_xd_vd_{v'}d_y=-1$. This contradicts the fact that $d_vd_{v'}=1$.\\
In the other case, when both $\mathcal{F}_{x\to y}$ and $\mathcal{F}_{x\to y}^c$ contain even mumber of unmatching edges, one can easily deduce that $d_xd_y=-1$ and using same technique we can get another contradiction. 

\end{itemize}

\end{proof}

Note that Corollary \ref{p1} and Theorem \ref{NH} give a full characterization of a unicyclic bipartite mixed graph with unique perfect matching where the inverse of its $\gamma$-hermitian adjacency matrix is $\{\pm 1\}$ diagonally similar to $\gamma$-hermitian adjacency matrix of a mixed graph. We summarize this characterization in the following corollary.

\begin{theorem}
Let $X$ be a unicyclic bipartite mixed graph with unique perfect matching and $H_\gamma$ its $\gamma$-hermitian adjacency matrix. Then, $H_\gamma^{-1}$ is $\pm 1$ diagonally similar to $\gamma$-hermitian adjacency matrix if and only if $X$ has more than two pegs and the cycle of $X$ contains even number of unmatching edges.
\end{theorem}


\section*{Acknowledgment}

The authors wish to acknowledge the support by the Deanship of Scientific Research at German Jordanian University.

\bibliography{InverseUni6}

\end{document}